\documentclass{conm-p-l}
\usepackage{latexsym, amsthm}

\usepackage{amsfonts, amsmath, amssymb}

\usepackage{hyperref}

\usepackage{euscript,mathrsfs}

\newtheorem{theorem}{Theorem}[section]
\newtheorem{lemma}[theorem]{Lemma}
\newtheorem{proposition}[theorem]{Proposition}
\newtheorem{conjecture}[theorem]{Conjecture}
\newtheorem{corollary}[theorem]{Corollary}

\theoremstyle{definition}

\newtheorem{example}[theorem]{Example}

\theoremstyle{remark}

\def\Fq{{\mathbb F}_q}

\def\d{{\delta}}
\def\AA{{\mathbb A}}
\def\FF{{\mathbb F}}
\def\PP{{\mathbb P}}
\def\dP{\widehat{\mathbb P}}

\def\hp3{\widehat{\mathbb P}^3}
\newcommand{\RM}{\mathrm{RM}}
\newcommand{\PRM}{\mathrm{PRM}}
\newcommand{\V}{\mathsf{V}}

\newcommand{\sP}{\mathsf{P}}
\newcommand{\Vmd}{\mathscr{V}_{m,d}}
\def\Aff{{\mathbb A}}
\def\Z{{\mathsf Z}}
\begin{document}

\title[Tsfasman-Boguslavsky Conjecture and Projective Reed-Muller codes]{Remarks on Tsfasman-Boguslavsky Conjecture and 
Higher~Weights of Projective Reed-Muller Codes}

\author{Mrinmoy Datta}
\address{Department of Mathematics,
Indian Institute of Technology Bombay,\newline \indent
Powai, Mumbai 400076, India.}
\curraddr{}
\email{mrinmoy.dat@gmail.com}
\thanks{The first named author is partially supported by a doctoral fellowship from the National Board for Higher Mathematics, a division of the Department of Atomic Energy, Govt. of India.}

\author{Sudhir R. Ghorpade}
\address{Department of Mathematics, 
Indian Institute of Technology Bombay,\newline \indent
Powai, Mumbai 400076, India.}
\email{srg@math.iitb.ac.in}
\thanks{The second named author is partially supported by Indo-Russian project INT/RFBR/P-114 from the Department of Science \& Technology, Govt. of India and  IRCC Award grant 12IRAWD009 from IIT Bombay.}


\date{}

\begin{abstract}

Tsfasman-Boguslavsky Conjecture predicts the maximum number of zeros that a system of linearly independent homogeneous polynomials of the same positive degree with coefficients in a finite field 
can have in the  corresponding projective space. We give a self-contained proof to show that this 
conjecture holds in the affirmative in the case of systems of three homogeneous polynomials, and also to show that the conjecture is false in the case of five quadrics in the 3-dimensional projective space 
over a finite field.  Connections between the Tsfasman-Boguslavsky Conjecture and the determination of generalized Hamming weights of projective Reed-Muller codes are outlined and these are also exploited 
to show that this conjecture holds in the affirmative in the case of systems of a ``large'' number of three 
homogeneous polynomials, and to deduce the counterexample of 5 quadrics. An application to the nonexistence of lines in certain Veronese varieties over finite fields is also included. 
\end{abstract}

\maketitle

\section{Introduction}
\label{sec:in}

Fix positive integers $m, d, r$  and a finite field  $\Fq$ with $q$ elements. Denote by $S$   the ring 
$\Fq[x_0,x_1, \dots , x_m]$ of polynomials in $m+1$ variables with coefficients in $\Fq$. 
For any integer $k$, let 
\begin{equation}
\label{pk}
p_k: = \begin{cases} q^k + q^{k-1} + \dots + q + 1 
& \text{ if } k \ge 0, \\
0 &  \text{ if } k<0. \end{cases} 
\end{equation}
Evidently, if $k\ge 0$, then $p_k$ is the number of points in $\PP^k(\Fq)$, the $k$-dimensional projective space over $\Fq$. Here is a remarkable conjecture mentioned in the title. 


{\bf Tsfasman-Boguslavsky Conjecture (TBC)}: Assume that 
$r \le \binom{m+d}{d}$ and $d < q-1$. Then 
the maximum number of common zeros in $\PP^m(\Fq)$ that a system of $r$ linearly independent homogeneous polynomials of degree $d$ in $S$ 
can have  is
\begin{equation}
\label{TBr}
T_r(d,m): = \displaystyle{   p_{m-2j} + \sum_{i=j}^m} \nu_i (p_{m-i} - p_{m-i-j}),
\end{equation}
where $(\nu_1, \dots, \nu_{m+1})$ is the $r$th element in descending lexicographic order among 
$(m+1)$-tuples $(\alpha_1, \dots , \alpha_{m+1})$ of nonnegative integers satisfying 
$\alpha_1+ \cdots + \alpha_{m+1} = d$, and where $j := \min\{i : \nu_i \ne 0\}$. 

The TBC, i.e., the above conjecture, has been shown to hold in the affirmative when $r=1$ by Serre \cite{Se} and independently, by 
S{\o}rensen  \cite{So}  in 1991 and when $r=2$ by Boguslavsky \cite{Bog} in 1997. Recently, in \cite{DG1} and \cite{DG} we proved the following. 

\begin{enumerate}
\item[{\bf 1.}] TBC 
holds in the affirmative when $r \le m+1$. 
\medskip

\item[{\bf 2.}]
TBC is false, in general, when $r > m+1$ and $d > 1$. More precisely,  
\eqref{TBr} is false for $d =2 $ and at least $\binom{m-1}{2}$ values of 
$r$ with $m+1 < r \le \binom{m+2}{2}$.
\end{enumerate}
Our proofs of {\bf 1} and {\bf 2} use, respectively, the following nontrivial theorems. 

\begin{description}
\item[Heijnen-Pellikaan Theorem] 
Assume that $d< q$ and $r\le m+1$. 
Then the maximum number of 
zeros in $\AA^m(\Fq)$ that 
a system of $r$ linearly independent  polynomials in $\Fq[x_1, \dots , x_m]$ of degree at most $d$ can have~is 
$$
(d-1)q^{m-1} + \lfloor q^{m-r} \rfloor.
$$ 

\medskip

\item[Zanella's Theorem] 
For any integer 
$j \ge -1$, write 
$
\delta_j : = 1 + 2 + \dots + (j+1)$.
Assume that 
$r \le \delta_m$ and let  $k$ be the unique integer such that 
$$
-1\le k < m \quad \text{and} \quad \delta_m - \delta_{k+1} < r \le \delta_{m} - \delta_{k}.
$$ 
If  $e_r(2,m)$ denotes  
the maximum number of zeros in $\PP^m(\Fq)$ for a system of $r$ 
linearly independent homogeneous 
polynomials in $S$
of degree $2$, then  
$$
e_r(2,m)  \le p_k + \lfloor q^{\epsilon -1} \rfloor, \quad \text{ where } \quad \epsilon: =  \delta_{m} - \delta_{k} - r.
$$
\end{description}
What is stated above are, in fact,  special cases of the results of Heijnen and Pellikaan \cite{HP}, 
which deals, more generally, with the case $r \le \binom{m+d}{d}$, and of Zanella \cite{Z}, which gives an exact value for $e_r(2,m)$. 
But even in these special cases, the results are nontrivial, and it would be interesting to have proofs of {\bf 1} and {\bf 2} that are independent of these nontrivial results. It is hoped that such proofs could pave the way toward a more general conjecture stated in \cite[Conjecture 6.6] {DG1} that 
ameliorates the TBC. With this in view, we give in this paper fairly self-contained proofs to show that: 
\begin{enumerate}
\item[{\bf 1.}] TBC 
holds in the affirmative when $r =3$ 
or $\binom{m+d}{d} - j$  for $j= 0, 1, \dots , d$. 
\smallskip

\item[{\bf 2.}]
TBC is false when $m=3$, $d=2$ and $r=5$, i.e., for $5$ quadrics in $\PP^3$. 
%
\end{enumerate}

The theorem of Heijnen-Pellikaan  is  intimately related to the determination of higher weights (also known as, generalized Hamming weights) of Reed-Muller codes $\RM_q(d,m)$ and in fact, that was the original motivation of 
\cite{HP}. In a similar manner, the TBC is closely related to determination of higher weights of 
projective Reed-Muller codes  $\PRM_q(d,m)$. Indeed, if we let $e_r(d,m)$ denote the 
maximum number of common zeros in $\PP^m(\Fq)$ for a system of $r$ 
linearly independent homogeneous  polynomials in $S$ of degree $d$, then  the $r^{\rm th}$ generalized Hamming weight of  the $q$-ary projective Reed-Muller code of order $d$ and length $p_m$ 
is given by
\begin{equation}
\label{erdr}
d_r\left(\PRM_q(d,m) \right) = p_m - e_r(d,m), \quad \text{ provided } d\le q. 
\end{equation}
In turn, coding theoretic results about $\PRM_q(d,m)$, such as the minimum distance of its dual, 
can be used to derive useful results concerning the TBC. This is the approach we take for constructing the smallest counterexample to TBC (of $5$ quadrics in $\PP^3$) as an alternative to using Zanella's theorem, and also for showing that the TBC holds in the affirmative for the last $d+1$ values of $r$. 
As for proving that the TBC holds in the affirmative when $r =3$, we tweak the arguments in \cite{DG1} 
to arrive at a proof that does not use the Heijnen-Pellikaan Theorem. By way of an application of our determination of some terminal higher weights of $\PRM_q(d,m)$   to finite geometry, we also show that Veronese varieties do not, in general, contain a projective line.  
In a section on preliminaries, we include a new and short proof of a basic bound due to Lachaud for the number of points of an equidimensional projective algebraic varieties of a given 
degree. 
This result was used in \cite{DG1} and will also be used here.

\section{Preliminaries}

We shall continue to use the notations and terminology introduced in the previous section. In particular, $r,m,d$ are fixed positive integers and $p_k$ is as in \eqref{pk}. Denote by $\overline{\FF}_q$ a fixed algebraic closure of $\Fq$.  
For any field $\FF$ and any nonnegative integer $j$, we will denote by 
$\PP^j(\FF)$ the $j$-dimensional projective space over $\FF$, and by 
$\dP^j(\FF)$ the dual of $\PP^j(\FF)$, consisting of all hyperplanes in 
$\PP^j(\FF)$. 
We are mostly interested in the case $\FF=\Fq$ and we often write $ \PP^j (\Fq)$ and $ \dP^j (\Fq)$ simply as $\PP^j$ and $\dP^j$, respectively. 
For any set $\mathscr{F}$ of homogeneous polynomials in $S:=\Fq[x_0,x_1. \dots , x_m]$, we denote by $\V(\mathscr{F})$ the set of common zeros in $\PP^m(\Fq)$ of the polynomials in $\mathscr{F}$. Likewise, for any 
$\mathscr{P} \subseteq \Fq[x_1. \dots , x_m]$, we denote by $\Z(\mathscr{P})$ the set of common zeros in $\AA^m(\Fq)$ of polynomials in $\mathscr{P}$. Sets such as $\V(\mathscr{F})$ and $\Z(\mathscr{P})$ are often referred to as projective varities and affine varieties, respectively. Thus we use the word \emph{variety} to  mean a Zariski closed subset in $\PP^m$ (or in $\AA^m$) that need not be irreducible. 
We call  such varieties to be \emph{irreducible} if the corresponding varieties over 
$\overline{\FF}_q$ 
are irreducible. Likewise for any (affine or projective) variety $X$, by the \emph{dimension} of $X$, denoted $\dim X$, we mean the dimension of the corresponding variety $\overline{X}$ of $X$ and for any projective variety $X\subseteq \PP^m(\Fq)$, by the \emph{degree} of $X$, denoted $\deg X$, we mean the degree of $\overline{X} \subseteq \PP^m(\overline{\FF}_q)$. Recall that if $\FF$ is an algebraically closed field,  $Y$  a projective variety in $\PP^m(\FF)$ 
and  $Y_1, \dots , Y_k$ the irreducible components of $Y$, then 
$$
\dim Y = \max \{\dim Y_i : i=1, \dots , k\} \quad \text{and} \quad 
\deg Y =\mathop{ \sum_{1\le i\le k}}_{\dim Y_i = \dim Y} \deg Y_i.
$$
In particular, if $Y$ is \emph{equidimensional}, i.e., if all its irreducible components have the same dimension, then $\deg Y$ is the sum of degrees of its irreducible components.

The following simple, but useful, lemma appears to be classical. We learned it from Zanella \cite{Z}. The short proof given below was suggested by M. Homma and is sketched in \cite[Remark 2.3]{DG}. For alternative proofs, one may refer to \cite{DG}. 

\begin{lemma}
\label{Zanella}
Let $X \subseteq \mathbb{P}^{m}$ and $a:=\max\{ |X \cap {\Pi}| : {\Pi} \in \dP^m\}$. 
Then $|X| \le a q + 1$.
\end{lemma}

\begin{proof}
Consider the incidence set 
${\mathscr{I}}:= \{ (P, {\Pi}) \in X \times \dP^m : P \in {\Pi} \} $. Clearly,
$$
|{\mathscr{I}}| = \sum_{P\in X} |\{ {\Pi}\in \dP^m : P\in {\Pi}\}| =   |X| p_{m-1} \quad \text{and} \quad  |{\mathscr{I}}| = \sum_{{\Pi}\in \dP^m} |X\cap {\Pi}|  \le a p_m.
$$
Also 
since $p_m = qp_{m-1} +1$ and $a\le p_{m-1}$, we see that 
$|X|  \le a p_m/ p_{m-1} \le a q+1$. 
\end{proof}

The following result is classical and appears, for example, as an exercise in Hartshorne \cite[Ex. I.1.8]{H}, whose solution is easily obtained using Krull's Principal Ideal Theorem. The assertion about the degree can also be proved directly, 
or it can be readily deduced from \cite[Thm. 7.7]{H}. 

\begin{lemma}\label{lemma:induc}
Let $\FF$ is an algebraically closed field and  $Z$ be an irreducible subvariety of $\PP^m(\FF)$. If $\Pi$ is any hyperplane in $\PP^m(\FF)$ such that $Z$ is not contained in $\Pi$, then $Z \cap \Pi$ is equidimensional of dimension $\dim Z - 1$, and moreover $\deg(Z\cap \Pi) \le \deg Z$. 
\end{lemma}
%

As an application of Lemmas  \ref{Zanella} and 
\ref{lemma:induc}, we prove a refined version of a theorem of Lachaud (cf. \cite[Prop. 12.1]{GL2}) for the number of points of projective varieties defined over $\Fq$. The statement and proof, as it appears in  \cite[Prop. 12.1]{GL2} is slightly erroneous, basically because the degree of a variety need not equal the sum of degrees of its irreducible components. But, 
as is noted in \cite{LR}, the result and the proof in \cite{GL2} is valid if the variety is assumed to be equidimensional. 
A counterexample in the non-equidimensional case is easily obtained by taking $X = \V(x_0x_1, x_0x_2)$ in $\PP^2$, i.e., $X=\V(x_0)\cup \V(x_1,x_2)$, which is the union of a (projective) line and a point outside it. A ``generic'' complementary dimensional linear subspace of $\PP^2$ meets $X$ in one point and thus $\deg X =1$ (alternatively, one can see this by computing 
the Hilbert polynomial of the 
homogeneous 
coordinate ring of $X$), whereas $|X(\Fq)|$ is clearly $q+2$, which is greater than $ 1\cdot p_1$. The proof below appears to be 
a little simpler than that in \cite[Prop. 12.1]{GL2} or \cite[Prop. 2.3]{LR}. 


\begin{proposition}
\label{LachaudBound}
Let $X \subset \PP^m(\overline{\FF}_q)$ be an equidimensional variety of degree $\d$ and dimension $s$. Then $$|X(\Fq)| \le \d p_s.$$
\end{proposition}
\begin{proof}
Induct on $m$. The case $m = 1$ being trivial, assume that $m>1$ and that the result holds for varieties in
$ \PP^{m-1}(\overline{\FF}_q)$. 
We divide the proof in three cases. 

{\bf Case 1: }  $X$ is contained in a hyperplane $\Pi  \in \dP^m(\overline{\FF}_q)$.

In this case, $X$  
is a variety in $\Pi \simeq \PP^{m-1}(\overline{\FF}_q)$ of dimension $s$ and degree $\d$. 
Hence by the induction hypothesis,
$|X (\Fq) | \le \d p_s.$ 

{\bf Case 2: } $X$ is irreducible and not contained in any hyperplane in $\PP^m(\overline{\FF}_q)$. 

In this case,  $X$ is nonempty (so that $s\ge 0$ and $\delta \ge 1$). 
Moreover, given any $\Pi \in \dP^m(\overline{\FF}_q)$,
by   Lemma \ref{lemma:induc}, we see that $X \cap \Pi$ is an equidimensional subvariety of $\Pi$
with $\dim X \cap \Pi = s - 1$ and $\deg X\cap \Pi \le \d$. 
So by induction hypothesis, $|(X \cap \Pi) (\Fq)| \le \d p_{s-1}$. 
Thus, Lemma \ref{Zanella} implies that  $|X(\Fq)| \le q (\d p_{s-1}) + 1 \le  \d p_s.$ 
 
{\bf Case 3: } $X$ is  an arbitrary variety in $\PP^m(\overline{\FF}_q)$ of dimension $s$ and degree $\d$.

In this case, write $X = X_1 \cup \dots \cup X_j \cup X_{j+1} \cup \dots \cup X_k$, where 
$X_1, \dots , X_k$ are the irreducible components of $X$  such that none among 
$X_1, \dots, X_j$ is contained in any hyperplane of $\PP^m(\overline{\FF}_q)$ whereas each of $X_{j+1}, \dots, X_{k}$ are contained in some hyperplane of $\PP^m(\overline{\FF}_q)$.  Since $X$ is equidimensional, 
$\dim X_i = s$ for all $i=1, \dots , k$. 
 From Case 1 and Case 2, we obtain 
$|X_i (\Fq)|  \le (\deg X_i) p_s$ for each $i=1, \dots k $. Consequently, 
$$
|X(\Fq)| \le \sum_{i=1}^k |X_i (\Fq)| \le  \left( \sum_{i=1}^k \deg X_i\right) p_s = \d p_s.
$$
This completes the proof.
\end{proof}

For ease of later reference, we record below optimal bounds for the number of $\Fq$-rational points of affine or projective hypersurfaces defined over $\Fq$. For a proof, 
one may refer to  \cite[p. 275]{LN} and  \cite[\S 2]{DG} or references therein. 

\begin{proposition}
\label{OreSerre}
Let $f\in \Fq[x_1, \dots , x_m]$ be 
of degree $d$ and let 
$F\in S$ be a nonzero homogeneous polynomial of degree $d$. 
Then 
\begin{enumerate}
\item[(i)]{\bf (Ore's Inequality)} If $d\le q$, then  $|\Z(f)| \le dq^{m-1}$. 

\item[(ii)]{\bf (Serre's Inequality)} If $d\le q+1$, then $|\V(F)| \le dq^{m-1} + p_{m-2}$. 
\end{enumerate}
\end{proposition}

%
%
%

\section{TBC for Systems of Three Polynomial Equations}

It is easy to see that when $r\le m+1$, the expression in 
\eqref{TBr} simplifies to 
$$
T_r(d,m)= \begin{cases} p_{m-r} & \text{ if } \, d = 1 \text{ and } \, 1\le r \le m+1, \\
(d-1)q^{m-1} + p_{m-2} + \lfloor q^{m-r}\rfloor & \text{ if } \, d > 1 \text{ and } \, 1\le r \le m+1. \end{cases}
$$
In the case of homogeneous linear polynomials, i.e., when $d=1$, it is obvious that the TBC is true. 
Moreover, 
it is not difficult to see that the TBC is also true when $m=1$. 
See, for example, \cite[\S 2.1]{DG1}. 
With this in view, we shall assume in this section that $d>1$ and $m>1$. 
For any nonnegative integer $j$,  denote by $S_j$ the $j^{\rm th}$ homogeneous component of $S =\Fq[x_0,x_1, \dots , x_m]$  consisting of homogeneous polynomials in $S$ of degree $j$ including the zero polynomial. 
The main result of this section is that the TBC  is true when  $r=3$. In other words, we prove the following. 

\begin{theorem}
\label{TBC3}
Assume that $m>1$ and $1< d < q-1$. 
Then the maximum number of common zeros in $\PP^m(\Fq)$ that a system of 
three linearly independent 
polynomials in $S_d$ 
can have  is
\begin{equation}
\label{TB3}
 (d-1)q^{m-1} + p_{m-2} + \lfloor q^{m-3}\rfloor. 
\end{equation}
\end{theorem}

As explained in the Introduction, this 
 is a special case of \cite[Theorem 6.3]{DG1}. But the proof given here does not use the Heijnen-Pellikaan Theorem (HPT). In fact, we follow a strategy similar to that   in \cite{DG1}
and 
give different proofs of those steps whose  proof 
in \cite{DG1} depended on HPT. 

Assume, as in Theorem \ref{TBC3}, that $m>1$ and $1< d < q-1$.  
Proceeding 
as in the proof of \cite[Lemma 6.1]{DG1}, we readily see that there do exist 3 linearly independent  polynomials in $S_d$ whose number of common zeros in $\PP^m(\Fq)$ is given by \eqref{TB3}. 
Now let  $F_1, F_2, F_3$ be arbitrary linearly independent homogeneous polynomials of degree $d$ in $S$. To complete the proof of Theorem \ref{TBC3} it suffices to show that 
\begin{equation}
\label{Tbound3}
|\V(F_1, F_2, F_3)| \le (d-1)q^{m-1}+q^{m-2} + \lfloor q^{m-3} \rfloor.
\end{equation}
To this end, fix a GCD (= greatest common divisor) $G$ of $F_1, F_2, F_3$ and a GCD $F_{ij}$ of $F_i$ and $F_j$ for $1\le i <  j\le 3$. Also  let $G_1, G_2, G_3 \in S$ be such that $F_i = G G_i$ for $i = 1, 2, 3$.  Evidently, $\gcd (G_1, G_2,  G_3) = 1$, that is, $G_1, G_2, G_3$ are 
coprime. But they may not be pairwise coprime. So we fix  a GCD $G_{ij}$ of $G_i$ and $G_j$ for $1\le i <  j\le 3$.  Note that 
$G, G_i, F_{ij}, G_{ij}$ are all homogeneous polynomials. Let
$$
b: = \deg G, \quad  b_{ij} := \deg F_{ij} \quad \mathrm{and} \quad c_{ij} := \deg G_{ij}  \quad  \mathrm{for} \quad 1 \le i <  j \le 3. 
$$ 
Note that $\deg G_i = d - b$ for $i = 1, 2, 3$ and $c_{ij}  = b_{ij} - b$ for $1\le i < j \le 3$. 
The proof is divided into the following three exclusive and exhaustive cases:
\begin{enumerate}
\item[{\bf (a)}]
$b_{ij} = 0$ for some $i, j \in \{1, 2, 3 \}$ with $i < j$.
\item[{\bf (b)}]
$1 \le b_{ij} < d-1$ for some $i, j \in \{1, 2, 3 \}$ with $i < j$.
\item[{\bf (c)}]
$b_{ij} = d-1$ for all $i, j \in \{1, 2, 3 \}$ with $i < j$.
\end{enumerate}
In cases {\bf (a)} and {\bf (b)}, 
Theorem \ref{TBC3}, 
can be proved using the basic bound of Lachaud (Proposition \ref{LachaudBound} of the previous section) as is shown in Lemmas 4.1 and 4.2 of \cite{DG1}. Thus we restrict ourselves to the harder case, which is Case {\bf (c)}. Here the proof in \cite{DG1} is based on Lemma 2.5 of \cite{DG1}, which separates out the circumstance where the $F_i$'s have a common linear factor, or equivalently, $G$ has a linear factor, and further based on Lemmas 5.3 and 5.4 of \cite{DG1}, which deal, respectively, with the subcases $b=d-1$ and $b< d-1$ of Case {\bf (c)} under the assumption that $G$ has no linear factor. Of these, Lemmas 2.5 and 5.4 of \cite{DG1} use the HPT in an essential way. So we will now prove versions of these lemmas without recourse to HPT when $r=3$. In the remainder of this section, we shall use the following notation:
$$
X := \V (F_1, F_2, F_3), \quad  X' := \V(G_1, G_2, G_3) \quad \mathrm{and} \quad Y := \V(G).
$$
Now here is a counterpart of  \cite[Lemma 5.3]{DG1} 
when $G$ does have a linear factor.

\begin{lemma}
\label{hc2}
With notations as above, suppose  
$b = d-1$ and  $G$ has a linear factor. Then \eqref{TB3} holds, i.e., $|X|\le (d-1)q^{m-1}+q^{m-2} + \lfloor q^{m-3} \rfloor$.
\end{lemma}

\begin{proof}
Since $b=d-1$ and since $F_1, F_2, F_3$ are linearly independent,  the factors $G_1, G_2, G_3$ are homogeneous linear as well as linearly independent. 
Thus  $X'$ is (isomorphic to) a projective space of dimension $m-3$. Moreover, if $H$ is a linear factor of $G$, then $Y\cap X'$ contains $\V(G_1, G_2, G_3, H)$, which 
is (isomorphic to) a projective space of dimension $\ge m-4$. Since $X=Y\cup X'$, 
applying part (ii) of Proposition \ref{OreSerre} (Serre's inequality) to 
$Y=\V(G)$, we obtain
\begin{equation*}
\begin{split}
|X|  = |Y \cup X'| & = |Y| + |X'| - |Y \cap X'| \\
& \le (d-1)q^{m-1} + p_{m-2} + p_{m-3} - p_{m-4} \\
& = (d-1)q^{m-1} + p_{m-2} + \lfloor q^{m-3} \rfloor, 
\end{split}
\end{equation*}
as desired. 
\end{proof}

To take care of Case {\bf (c)} when $b<d-1$ and regardless of whether or not $G$ has a linear factor, we first need a reduction given by the following. 

\begin{lemma}
\label{reduction}
With notations as above, suppose $b_{ij} = d-1$ for all $i, j \in \{1, 2, 3 \}$ with $i < j$ and suppose $b< d-1$. Then $b=d-2$ and there exist homogeneous linear polynomials $H_1, H_2, H_3 \in S$ with
 no two of $H_1, H_2, H_3$ differ by a constant (in $\Fq$) such that $G_1 = H_1 H_2, G_2 = H_2 H_3$ and $G_3 = H_3 H_1$. 
\end{lemma}

\begin{proof}
Write $b=d-k$ and note that $k\ge 2$. 
We claim that each $G_i$ is a product of $k$ linear factors in $S_1$, no two of which differ by a constant. 
Indeed, if some  $G_i$ had an irreducible factor $Q\in S$ with $\deg Q \ge 2$, then since $\deg G_i = k$ and  the degree of a GCD of $G_i$ and $G_j$ is $k-1$ for each $j\ne i$, it follows that $Q$ divides $G_j$ for all $j\ne i$. But then this contradicts the fact that $G_1, G_2, G_3$ are coprime. In a similar manner if $H^2$ divides $G_i$ for some $i$ and some nonzero $H\in S_1$, then $H$ divides $G_j$ for all $j\ne i$, again leading to a contradiction. This proves the claim. 

Thus we can write $G_1 = H_1 \dots H_{k-1} H_k$ and $G_2 = H_1 \dots H_{k-1} H_{k+1}$ where $H_1, \dots , H_{k+1} \in S_1$ and no two of them differ by a constant. 
Suppose, if possible, $b< d-2$, i.e., $k\ge 3$. Then at least one among $H_1, \dots H_{k-1}$ must be a factor of $G_3$. This contradicts the fact that $G_1, G_2, G_3$ are coprime. Hence we must have $k=2$, i.e., $b=d-2$, and moreover $G_1 = H_1 H_2, G_2 = H_2 H_3$ and $G_3 = H_3 H_1$. 
\end{proof}

In view of the discussion and the results above, it remains to prove the following. 

\begin{lemma}
\label{hc1}
With notations as above, suppose  $b_{ij} = d-1$ for all $i, j \in \{1, 2, 3 \}$ with $i < j$ and $b < d-1$. 
Then \eqref{TB3} holds; in fact, $|X| < (d-1)q^{m-1}+q^{m-2} + \lfloor q^{m-3} \rfloor$.
\end{lemma}

\begin{proof}
By Lemma \ref{reduction}, $b=d-2$ and moreover, 
$G_1 = H_1 H_2$, $G_2 = H_2 H_3$ and $G_3 = H_3 H_1$ for some $H_1, H_2, H_3 \in S_1$, no two of  which differ by a constant. 
We will now estimate $|X'|$ by considering two cases. 

{\bf Case 1: } $H_1, H_2, H_3$ are linearly dependent.

By assumption, any two among $H_1, H_2, H_3$ are linearly independent. So 
in this case, we can write  $H_3 = \lambda H_1+\mu  H_2$ for some $ \lambda, \mu \in \Fq$ with $ \lambda\ne 0$ and $\mu \ne 0$.~Hence 
$$
X' = \V(H_1H_2, H_2H_3, H_3H_1) = \V(H_1, \mu  H_2^2) \cup \V(H_2,  \lambda H_1^2) = \V(H_1, H_2).
$$
In particular, 
$|X'| = p_{m-2}$.

{\bf Case 2: } $H_1, H_2, H_3$ are linearly independent.

By a linear change of coordinates, we may assume that $H_1=x_0$, $H_2=x_1$ and $H_3=x_2$. 
 Let $P = [c_0: c_1: \dots : c_m] \in \PP^m$ be a common zero of $x_0x_1, x_1 x_2, x_2 x_0$. 
Considering the 
possibilities (i) $c_0=1$, (ii) $c_0=0$ and $c_1=1$, and (iii) $c_0=c_1=0$ separately, we 
see that $P$ can be chosen in exactly $q^{m-2} + q^{m-2} + p_{m-2}$ ways. Thus in this case, 
 $|X'| = 2q^{m-2} + p_{m-2}$. 

%

Thus we always have $|X'| \le 2q^{m-2} + p_{m-2}$. Also, as a consequence of Serre's inequality we have $|Y| \le (d-2) q^{m-1} + p_{m-2}$. Hence 
\begin{equation*}
\begin{split}
|X| & \le |X'| + |Y| \\
& \le 2q^{m-2} + p_{m-2} + (d-2) q^{m-1} + p_{m-2} \\
& = (d-2) q^{m-1} + 2q^{m-2} + 2 p_{m-2} \\
\end{split}
\end{equation*} 
The lemma will be proved if we show that 
$$
(d-1)q^{m-1}  + p_{m-2} + \lfloor q^{m-3} \rfloor > (d-2) q^{m-1} + 2q^{m-2} + 2 p_{m-2}.
$$
To this end, let us 
consider the difference of the above two expressions. 
\begin{equation*}
\begin{split}
& (d-1)q^{m-1}  + p_{m-2} + \lfloor q^{m-3} \rfloor  -\big( (d-2) q^{m-1} + 2q^{m-2} + 2 p_{m-2} \big) \\
 = & \ q^{m-1} - p_{m-2} - 2q^{m-2} + \lfloor q^{m-3} \rfloor \\
 \ge &  \ \frac{q^{m-1}(q-1) - (q^{m-1} -1) + 2q^{m-2} (q-1)}{q-1}  + \lfloor q^{m-3} \rfloor \\
 = & \ \frac{q^m - 4 q^{m-1} + 2 q^{m-2}}{q-1}   + \lfloor q^{m-3} \rfloor .\\
\end{split}
\end{equation*}
The above quantity is strictly positive 
since $q \ge 4$, thanks to the assumption $d>1$ and $1<d<q-1$.
This completes the proof.  
\end{proof}

{\sc Proof of Theorem \ref{TBC3}.} In view of the earlier discussion, the theorem follows from Lemmas  \ref{hc2} and \ref{hc1} together with Lemmas 4.1, 4,2, 5.3 and 6.1 of \cite{DG1}. \qed

\medskip

We remark that the above proof can also be adapted to the case of two linearly independent homogeneous polynomials. This would, in fact, be a somewhat simpler way of proving Boguslavsky's theorem \cite[Theorem 2]{Bog}. At any rate,  Theorem \ref{TBC3} together with Boguslavsky's theorem and Serre's inequality shows that 
\begin{equation}
\label{erupto3}
e_r(d,m) = (d-1)q^{m-1}  + p_{m-2} + \lfloor q^{m-r} \rfloor  \quad \text{ for } 1\le r \le 3 \text{ and } 1 < d < q-1.
\end{equation}

\begin{corollary}[Special Case of HPT] Let $\delta$ be a positive integer $ < q-2$. 
For any linearly independent polynomials  $f_1, f_2, f_3 \in \Fq[x_1, \dots , x_m]$ of degree $\le \d$,  
$$
|\Z(f_1, f_2, f_3)|  \le (\d-1) q^{m-1}  + \lfloor q^{m-3} \rfloor.
$$
\end{corollary}

\begin{proof}
Homogenize $f_1, f_2 , f_3$ using the extra variable $x_0$  to obtain $3$ linearly independent polynomials, say $F^*_1, F^*_2 , F^*_3$, in $S_{\d}$. Let $\widetilde{F}_i := x_0F^*_i$ for $i=1, 2,3$. Using Theorem \ref{TBC3}  applied to be $\widetilde{F}_1, \widetilde{F}_2, \widetilde{F}_3$ in $S_{\d+1}$, we see that 
$$
| \V(\widetilde{F}_1, \widetilde{F}_2, \widetilde{F}_3)| \le \d q^{m-1}  + p_{m-2} + \lfloor q^{m-3} \rfloor .
$$ 
On the other hand, intersecting $\V(\widetilde{F}_1, \widetilde{F}_2, \widetilde{F}_3)$ with the hyperplane $\V(x_0)$ and its complement, we find that 
$
| \V(\widetilde{F}_1, \widetilde{F}_2, \widetilde{F}_3)| = p_{m-1} + \left| \Z(f_1, f_2 , f_3)\right|$.  
It follows that $|\Z(f_1, f_2 , f_3)| \le  (\d-1) q^{m-1}  + \lfloor q^{m-3} \rfloor$. 
\end{proof}

We remark that a similar argument can be used to derive the Heijnen-Pellikaan bound for two linearly independent polynomials using Boguslavsky's theorem \cite[Theorem 2]{Bog} and to derive Ore's inequality from Serre's inequality (Proposition \ref{OreSerre}). 

\section{Projective Reed-Muller codes and their higher weights}
%
%
The notion of higher weights (also known as generalized Hamming weights) of a linear code is now fairly well-known and we refer to \cite{W} for 
basic definitions and results. Note that \cite[Sec. 4]{DG} also gives a quick recap. 
The following basic  result of  Wei \cite{W} will be useful to us.

\begin{proposition}
\label{W}
Let $C$ be an $[n, k]_q$-code. 
Let $d_i$ 
be the $i^{\rm th}$ higher weight of $C$ and  $d_j^{\perp}$ the $j^{\rm th}$ higher weight of 
 $C^{\perp}$, for $i=1, \dots, k$ and $j = 1,  \dots, n - k$. 
Then 
\begin{itemize}
\item[(i)]
(Monotonicity) $1 \le d_1 < d_2 < \dots < d_k \le n$.
\item[(ii)]
(Duality) $\{d_j^{\perp} : j = 1, 2, \dots, n -k\} = \{1, \dots, n\} \setminus \{d_i : i = 1, 2, \dots, k \}$.
\end{itemize}
\end{proposition}

We shall now discuss two important examples of linear codes that are relevant to us. The first is quite classical and we touch upon it only briefly. As before, positive integers $d$ and $m$ are kept fixed throughout.  

\subsection{Reed-Muller Codes} 
Let $n=q^m$ and let $P_1, \dots, P_{n}$ be an ordered listing of elements of $\Aff^m(\Fq)$, i.e., $\Fq^m$. 
The \textit{Reed-Muller code} of order $d$ and length $q^m$ is 
denoted by $\RM_q (d, m)$ and defined by 
$$
\RM_q(d,m):=\{ \left( f(P_1), \dots , f(P_n)\right) : f\in \Fq[x_1, \dots , x_m], \; \deg f \le d\}.
$$
These codes are also known as generalized Reed-Muller codes or affine Reed-Muller codes. A summary of many of their properties is given in \cite[Prop. 4]{BGH2}. In particular, we note that if $d<q$, then the dimension of $ \RM_q(d,m)$ is $\binom{m+d}{d}$ and, as in \eqref{erdr}, the higher weights are related to an affine analogue of $e_r (d, m) $: 
\begin{equation}
\label{drATBC}
d_r\left(\RM_q(d,m) \right) = q^m - e^{\Aff}_r (d, m) ,  \quad \text{ provided } d < q, 
\end{equation}
where $e^{\Aff}_r (d, m)$ 
denotes the maximum number of zeros in $\Aff^m(\Fq)$ that a system of $r$ linearly independent polynomials in $\Fq[x_1, \dots , x_m]$ of degree $\le d$ can have. 
Here is an easy 
consequence of Ore's inequality and monotonicity of higher weights. 

\begin{lemma} 
\label{erA}
Let $k:= \binom{m+d}{d}$ and let $e^{\Aff}_r(d, m)$ be as defined above. Then
$$
e^{\Aff}_r(d, m) \le dq^{m-1} - r + 1 \quad \text{whenever} \quad 
1\le r \le k 
\text{ and } 1\le d < q. 
$$
\end{lemma}

\begin{proof}
By part (i) of Proposition \ref{OreSerre}, 
$e^{\Aff}_1(d,m) \le dq^{m-1}$, whereas by \eqref{drATBC} and part  (i) of Proposition \ref{W}, 
$e^{\Aff}_1(d, m) > e^{\Aff}_2(d, m) > \dots > e^{\Aff}_k(d, m) \ge 0$. 
Hence it follows that 
$e^{\Aff}_r(d, m) \le e^{\Aff}_1(d, m) -r + 1 \le   dq^{m-1} - r + 1$. 
\end{proof}

It may be remarked that the Heijnen-Pellikaan Theorem (HPT) is a much more general result than Lemma \ref{erA} above and gives a complete description of $e^{\Aff}_r(d, m)$. But of course HPT is more difficult to prove, while the proof above is almost trivial. 

\subsection{Projective Reed-Muller Codes}
Let $n = p_m$. We know that each point of $\PP^m(\Fq)$ admits a unique representative in $\Fq^{m+1}$ in which the first nonzero coordinate is $1$. Let 
$\sP_1, \dots, \sP_{n}$ be an ordered listing of such representatives in $\Fq^{m+1}$ of points of $\PP^m(\Fq)$.
The \textit{projective Reed-Muller code} of order $d$ and length $p_m$ is 
denoted by $\PRM_q (d, m)$ and defined by 
$$
\PRM_q(d,m):=\{ \left( F(\sP_1), \dots , F(\sP_n)\right) : F\in S_d\}.
$$
The following two 
results due to S{\o}rensen \cite[Theorems 1 and 2]{So} describe some of the fundamental properties of projective Reed-Muller codes.

\begin{proposition} 
\label{so1}
Let $1 \le d \le m(q-1)$.
 Then the projective Reed-Muller code $\PRM_q(d,m)$ is a nondegenerate $[n, k, d']_q$-code with
\begin{itemize}
\item[(i)]
$n = p_m$, 
\item[(ii)]
$k = \displaystyle{\sum_{\substack{t = 1 \\ t \equiv d \pmod {q-1}}}^{d}} \left (\displaystyle{\sum_{j=0}^{m+1}} (-1)^j {m+1 \choose j} {t - jq + m \choose t - jq}\right)$, 
\item[(iii)]
$d' = (q - s)q^{m-t-1},$ where $s$ and $t$ are unique integers satisfying 
$$d - 1 = t (q-1) + s \quad \mathrm{ and } \quad 0 \le s < q-1.$$
\end{itemize}
\end{proposition}
\begin{proposition} 
\label{so2}
Let $1 \le d \le m(q-1)$ be such that $q - 1 \nmid d$. Then the dual of the projective Reed-Muller code $\PRM_q(d, m)$ is also a projective Reed-Muller code $\PRM_q(d^{\perp}, m)$, where $d^{\perp} = m(q-1) - d$.  
\end{proposition}

\begin{corollary}
\label{done}
If $1 \le d<q-1$, then the minimum distance of $\PRM_q (d, m)^{\perp}$ is $d+2$.
\end{corollary}
\begin{proof}
By Proposition \ref{so2},  we see that $\PRM_q (d, m)^{\perp} = \PRM_q (d^{\perp}, m)$, where $d^{\perp} = m(q-1) - d$. Now $$d^{\perp} - 1 = m(q-1) - d - 1 = 
(m-1)(q-1) + (q-d-2) 
$$
Thus the minimum distance of $\PRM_q (d, m)^{\perp}$ is $(q - q + d + 2)q^{m - m + 1 - 1} = d + 2$. 
\end{proof}

 The Tsfasman-Boguslavsky conjecture is very closely related to the generalized Hamming weights or higher weights of projective Reed-Muller codes. Indeed one observes that 
 if $d \le q$, then for  
$1\le r \le \binom{m+d}{d}$,  
\begin{equation}
\label{drTBC}
d_r\left(\PRM_q(d,m) \right) = p_m - e_r (d, m) , 
\end{equation}
where $e_r(d,m)$ is as defined in the Introduction. For more on this relation, one may refer to \cite{Bog, DG, TV2}. We now derive a useful consequence of Wei duality and monotonicity of higher weights of a linear code.  

%
%

\begin{corollary}
\label{terminal}
Suppose $1 \le d < q-1$ and  let $k:=\binom{m+d}{d}$. Then,
$$
d_{k - s} (\PRM_q (d, m)) = n - s \quad \text{ for } s = 0, 1, \dots, d.
$$
\end{corollary}

\begin{proof}
Let $d_1^{\perp}$ denote the minimum distance of $\PRM_q (d, m)^{\perp}$. 
The duality part of Theorem \ref{W} implies that all the integers in the interval $(n + 1 - d_1^{\perp}, \, n]$ will be attained by some higher weights of $\PRM_q(d, m)$. 
Hence by Corollary \ref{done},
 the integers $n-d, \, n-d+1, \dots, n$ are among the higher weights of $\PRM_q (d, m)$. 
By the nondegeneracy of $\PRM_q(d, m)$, we have $d_k (\PRM_q (d, m)) = n$. 
Now the monotonicity of higher weights (Theorem \ref{W}) implies  
$d_{k-s} (\PRM_q(d, m)) = n - s$ for $s = 0, 1, \dots, d$.
\end{proof}

\begin{theorem}
\label{terminale}
Suppose $1 \le d < q-1$ and  let $k:=\binom{m+d}{d}$. Then 
$$
e_{k-s} (d, m) = s = T_{k-s}(d,m) \quad  \text{for } 
s = 0, 1, \dots, d.
$$
Consequently, the 
Tsfasman-Boguslavsky Conjecture holds in the affirmative for the last $d+1$ values of $r$, i.e., when 
$r=k - s$ 
 for $s= 0, 1, \dots , d$. 

\end{theorem} 

\begin{proof}
From \eqref{drTBC} and Corollary \ref{terminal}, 
we obtain $e_{k-s} (d, m) = s$ for $s = 0, 1, \dots, d$. 
Further, it is clear that  
if the $(m+1)$-tuples $(\alpha_1, \dots , \alpha_{m+1})$ of nonnegative integers satisfying 
$\alpha_1+ \cdots + \alpha_{m+1} = d$ are ordered  in descending lexicographic order, then 
the last $d+1$ such tuples are $(0,0, \dots , 0, d, 0), (0,0, \dots , 0, d-1, 1), \dots, (0,0, \dots , 0, 0, d)$. 
In particular, for $s = 0, 1, \dots, d$, the $(k-s)^{\rm th}$ tuple is $(0,0, \dots , 0, s, d-s)$ and so the 
Tsfasman-Boguslavsky bound in \eqref{TBr} 
is 
$T_{k-s}(d,m) = p_{m-2m} + s (p_0 - p_{-m}) = s$.
\end{proof}

\begin{example}
\label{z5}
Let us consider a small case where $m=2$ and $d = 2$. Here $k = 6$. 
Assume that $q \ge 4$. Then from \eqref{erupto3}, we see that 
$$
e_1(2,2) = 2q + 1, \quad e_2(2,2) = q + 2, \quad \text{ and } \quad e_3(2,2) = q + 1,
$$
whereas by Corollary \ref{terminale}, we see that 
$$
e_4(2,2) = 2, \quad e_5(2,2) = 1, \quad \text{ and } \quad e_6(2,2) = 0.
$$
Using \eqref{drTBC}, we can also compute $d_r(\PRM_q(2,2))$ for $r=1, \dots , 6$. 
\end{example}


To end this section, we note that a special case of Corollary \ref{terminal} or Theorem \ref{terminale} gives rise to seemingly nontrivial geometric results such as the following. 

\begin{corollary}
\label{line}
Assume that $1 < d < q-1$. Let $k:=\binom{m+d}{d}$ and let $\Vmd$ 
be the Veronese variety given by the image of $\PP^m(\Fq)$ in $\PP^{k-1}(\Fq)$ under the Veronese map of degree $d$. Then $\Vmd$ does not contain a (projective) line in $\PP^{k-1}(\Fq)$. 
\end{corollary}

\begin{proof}
Let $n:=p_m$. Since $d>1$, we see from Corollary \ref{terminal} that 
\begin{equation}
\label{eq1}
d_{k - 2} (\PRM_q (d, m)) = n - 2.
\end{equation}
On the other hand, using the one-to-one correspondence between $[n,k]_q$-codes and $[n,k]_q$-projective systems (see, for example, equation (4) and Theorem 2.1 of \cite{TV2}), we know that 
$d_{r} (\PRM_q (d, m))$ is equal to 
\begin{equation}
\label{eq2}
 n - \max\left\{|\Vmd\cap \Pi| : \Pi \text{ projective subspace of codimension $r$ in } \PP^{k-1}(\Fq)\right\}.
\end{equation}
Indeed $\Vmd$ is the projective system corresponding to $\PRM_q(d,m)$ when $d\le q$ and this projective system is evidently nondegenerate. In particular, when $r=k-2$, a projective subspace of codimension $r$  in $\PP^{k-1}(\Fq)$ correspond to a (projective) line in $\PP^{k-1}(\Fq)$ and if $\Vmd$ were to contain a line, then \eqref{eq2} would be equal to $n-(1+q)$, which is strictly smaller than $ n-2$, in contradiction to \eqref{eq1}. 
\end{proof}

It appears difficult to locate a result in Corollary \ref{line} 
in standard treatises on finite geometry such as \cite{HT}. It can, however, be deduced from a relatively recent result of Kantor and Shult \cite[Thm. 1.1]{KS}. 
That the Veronese variety does not usually contain a line should be contrasted with  the fact that the Grassmann variety $G_{\ell,m}$ not only contains plenty of lines, but also projective linear suspaces of dimension $\max\{\ell, m-\ell\}$; see, for example, \cite[\S 24.2]{HT} or \cite[\S 2]{GPP}.

\section{A Counterexample to the TBC} 

The aim of this section is to show that the TBC is false when  $m=3$, $d=2$ and $r=5$, i.e., for $5$ quadrics in $\PP^3$. Note that for these values of $m,d,r$, the Tsfasman-Boguslavsky bound \eqref{TBr} works out to be 
$$
T_5(3,2) = 2(q+1).
$$
The following result shows 
that this isn't the true maximum. i.e., $T_5(3,2) \ne e_5(3,2)$, 
Note that the condition $d<q-1$ in the TBC translates to $q>3$ when $d=2$. 

\begin{theorem}
\label{Z5}
Assume that $q>3$. 
Let $F_1, \dots, F_5 $ be linearly independent homogeneous polynomials of degree $2$ in $\Fq [x_0, x_1, x_2, x_3]$. 
Then 
$$|\V(F_1, \dots, F_5)| \le 2q + 1.$$ 
\end{theorem} 

\begin{proof}
Let $X = \V(F_1, \dots, F_5) \subseteq \PP^3$.
If the restrictions  $F_1|_{\Pi}, \dots, F_r|_{\Pi}$ are linearly independent for every 
plane $\Pi \in \hp3$, then 
by Example \ref{z5}, we obtain 
$$
 |X \cap \Pi| \le e_5(2,2) = 1 \quad \text{ for every } \Pi \in \hp3.
$$ 
Hence Lemma \ref{Zanella} shows that $|X| \le q + 1 \le 2q + 1.$ 

Now let us suppose there exists 
$\Pi \in \hp3$ such that $F_1|_{\Pi}, \dots, F_5|_{\Pi}$ are linearly dependent. 
By a projective linear change of coordinates, if necessary, we may assume that $\Pi = V(x_0)$. Let 
$\mathscr{F}$ denote the $\Fq$-vector subspace of $\Fq [x_0,x_1, x_2, x_3]$ generated by $F_1, \dots, F_5$, and let 
$$
\mathscr{F}_{\Pi} := \{F \in \mathscr{F} : x_0 \mid F\} , \quad \ s = \dim \mathscr{F}_{\Pi} \quad \text{ and } \quad t:= 5 - s.
$$ 
Since a nontrivial $\Fq$-linear combination of $F_1, \dots , F_5$ vanishes on $\Pi$, we find $s\ge 1$. Also elements of $\mathscr{F}_{\Pi}$ are of the form $x_0L$, where $L$ is a homogeneous linear~polynomial in $\Fq[x_0,x_1, x_2, x_3]$. The space of such  homogeneous linear polynomials has dimension $4$, and so $s\le 4$. Hence $1 \le t \le 4$. Choose a basis $F^*_1, \dots, F^*_5$ of $\mathscr{F}$ such that $F^*_{t+1}, \dots, F^*_5 \in \mathscr{F}_{\Pi}$. Evidently,  
$ \V(F_1, \dots, F_5) = \V(\mathscr{F}) =  \V(F^*_1, \dots, F^*_5)$. 
Let $G_i = F^*_i|_{\Pi}$ and $f_i = F^*_i|_{\Pi^c}$ for $i = 1, \dots, 5$, where $\Pi^c$ denotes the complement of $\Pi$ in $\PP^m$. 
Thus  $G_i$'s are obtained from $F^*_i$'s by putting $x_0 = 0$, while $f_i$'s  are obtained from $F^*_i$'s by putting $x_0 = 1$. 
Note that 
$G_1, \dots, G_t$ are linearly independent homogeneous polynomials in $\Fq [x_1, x_2, x_3]$ of degree $2$, while $G_{t+1} = \dots = G_5 = 0$, whereas  $f_{1}, \dots , f_5$ are  linearly independent (possibly nonhomogeneous)  polynomials  in $\Fq [x_1, x_2, x_3]$ with $\deg f_i = 2$ for $i=1, \dots ,t$ and $\deg f_i \le 1$ for $i= {t+1}, \dots , 5$. 
We now make a case-by-case analysis.  

\medskip

\textbf{Case 1: } $t = 1$. 

Since $f_2, f_3, f_4, f_5$ are linearly independent elements of the vector space of polynomials  in $\Fq [x_1, x_2, x_3]$  of degree  $\le 1$ and since this vector 
space has dimension~$4$, we see that $1$ is a linear combination of $f_2, f_3, f_4, f_5$. Consequently, $|X \cap \Pi^c| \le |Z(f_2, f_3, f_4, f_5)| = 0$. On the other hand, from Example \ref{z5}, 
we see that  
 $|X \cap \Pi| = |\V (G_1)| \le e_1 (2,2)= 2q+1$.  It follows that $|X| \le 2q + 1$. 

\medskip

\textbf{Case 2: }  $t = 2$.

Here,  $f_3, f_4, f_5$ are three linearly independent (possibly nonhomogeneous) linear polynomials in 3 variables. Hence the system of linear equations $f_3 = f_4 = f_5 = 0$ can have at most $1$ solution. Thus $|X \cap \Pi^c|  \le 1$. On the other hand, Example~\ref{z5} shows that $|X \cap \Pi| = |\V(G_1, G_2)| \le e_2(2,2) = q + 2$. 
Consequently,  $|X| \le q+3$. 

\medskip

\textbf{Case 3: } $t = 3$. 

By a similar argument as in Case 2, we observe that $|X \cap \Pi^c| \le q$. 
On the other hand,  $|X \cap \Pi| = |\V(G_1, G_2, G_3)| \le e_3(2,2) = q+1$. Thus, $|X| \le 2q + 1$. 
 
\medskip

\textbf{Case 4: } $t = 4$. 

Here, $|X \cap \Pi| = |\V(G_1, G_2, G_3, G_4)| \le e_4(2,2) = 2$ and so 
it suffices to show that $|X \cap \Pi^c| = |\Z(f_1, f_2, f_3, f_4, f_5)| \le 2q-1$. To this end, first note that if $\deg f_5=0$, then  $ |\Z(f_1, f_2, f_3, f_4, f_5)| =0$, and we are done. Thus we may assume that $f_5$ is 
of degree $1$ in $\Fq[x_1, x_2, x_3]$ and so $F^*_5 = x_0(a_0x_0 + a_1x_1+a_2x_2+a_3x_3)$ for some
$a_0,a_1, a_2, a_3\in \Fq$ with $(a_1, a_2, a_3)\ne (0,0,0)$.  By a homogeneous linear change of variables leaving $x_0$ unchanged, we can assume without loss of generality that $F_5^* = x_0x_1$; in particular, $f_5 = x_1$. Now let ${\mathscr{P}}$ 
be the $\Fq$-vector subspace of $\Fq[x_1, x_2, x_3]$ generated by $f_1, \dots , f_5$. Note that substituting $x_0=1$ gives an 
isomorphism of $\mathscr{F}$ onto $\mathscr{P}$.  Further, if we let 
$$
\mathscr{P}_{1} := \{f \in \mathscr{P} : x_1 \mid f\} , \quad \ s' = \dim \mathscr{P}_{1} \quad \text{ and } \quad t':= 5 - s', 
$$ 
then $\mathscr{P}_{1}$ is isomorphic to $\mathscr{F}_{1} := \{F\in \mathscr{F} : x_1 \mid F\}$.  Also as 
in the case of $\mathscr{F}_{\Pi}$, we see that $1\le s' \le 4$. Moreover if $s'=4$, i.e., if $t'=1$, then arguing as in Case~1 above, but with $x_0$ and $x_1$ interchanged, we directly find $|X| \le 2q + 1$. Thus suppose $t'>1$. Replacing $f_1, \dots , f_5$ by a suitable basis of $\mathscr{P}$, we may suppose that $f_{t'+1}, \dots , f_{5}$ constitute a  basis of $\mathscr{P}_{1}$ and no nontrivial linear combination of 
$f_1, \dots , f_{t'}$ is in $\mathscr{P}_{1}$. Now if we write $f_j = x_1g_j + h_j$ for unique $g_j \in \Fq[x_1, x_2, x_3]$ of degree $\le 1$ and $h_j\in \Fq[x_2, x_3]$ of degree $\le 2$ ($1\le j \le t'$), then $h_1, \dots , h_{t'}$ are linearly independent  and if $\Z(h_1, \dots ,  h_{t'})$ denotes their zero set in $\Aff^2$, then 
$$
|\Z(f_1, f_2, f_3, f_4, f_5)| = |\Z(h_1, \dots ,  h_{t'})| \le e^{\Aff}_{t'}(2,2) \le 2q - t'+1 \le 2q-1,  
$$
where the last two inequalities follows from Lemma \ref{erA} and the fact that $t'>1$. This completes the proof. 
\end{proof}

It would be interesting to determine $e_r(d,m)$ for all permissible values of $r,d,m$, especially when $r\ge m+1$ and $d\le q$. In \cite{DG1} a conjecture is made for $r\le { {m+d-1}\choose{m} }$ and $d<q$, but this is open,
in general.

\section*{Acknowledgments}
\label{secAck}
\begin{small}
We are grateful to the anonymous referee for a careful reading of a preliminary version of this paper
and helpful comments. We would also like to thank Luca Giuzzi for bringing \cite{KS} to our attention. 
\end{small}

\end{document}